\documentclass[a4paper,10pt]{article}

\usepackage{amsmath}
\usepackage{mathtools}
\usepackage{amssymb}
\usepackage{bbm}
\usepackage{enumerate}
\usepackage{algorithm}
\usepackage{algorithmic}
\usepackage{amsfonts}
\usepackage{amsthm}
\usepackage{bbm}
\usepackage{bm}
\usepackage[mathscr]{eucal}
\usepackage{a4wide}
\usepackage{authblk}
\usepackage{tikz}
\usetikzlibrary{decorations.pathreplacing}
\usepackage{hyperref}

\newcommand{\pr}{\mathbb{P}}								

\newcommand{\Prob}[1]{\pr\left(#1\right)}					
																
\newcommand{\e}{\mathbb{E}}								

\newcommand{\Exp}[1]{\e\left[#1\right]}					
															
\newcommand{\bigO}[1]{O\left(#1\right)}				

\DeclarePairedDelimiter{\abs}{\lvert}{\rvert}
\DeclarePairedDelimiter{\abss}{\|}{\|}

\newcommand{\tm}{T_{\max}}

\newcommand{\ep}{\varepsilon}

\newcommand{\Ecal}{\ensuremath{\mathcal{E}}}			

\newcommand{\Gbb}{\ensuremath{\mathbb{G}}}				
\newcommand{\Gcal}{\ensuremath{\mathcal{G}}}			
\newcommand{\Gcals}{\Gcal^{*}}
\newcommand{\Gs}{G^{*}}

\definecolor{edgered}{RGB}{255,0,0}
\definecolor{edgeblue}{RGB}{20,40,255}
\definecolor{edgegreen}{RGB}{40,190,70}

\allowdisplaybreaks

\newtheorem{theorem}{Theorem}[section]
\newtheorem{lemma}[theorem]{Lemma}

\newtheorem{remark}[theorem]{Remark}

\theoremstyle{definition}
\newtheorem{definition}[theorem]{Definition}
\newtheorem{claim}[theorem]{Claim}

\title{Regular graphs with many triangles are structured}
\author[1]{Pim van der Hoorn}
\author[2]{Gabor Lippner}
\author[3,4]{Elchanan Mossel}

\affil[1]{TU Eindhoven, Department of Mathematic and Computer Science}
\affil[2]{Northeastern University, Department of Mathematics}
\affil[3]{MIT, Department of Mathematics}
\affil[4]{MIT Institute for Data Systems and Society}

\begin{document}

\maketitle

\abstract{We compute the leading asymptotics of the logarithm of the number of $d$-regular graph having at least a fixed positive fraction $c$ of the maximum possible number of triangles, and provide a strong structural description of almost all such graphs.

When $d$ is constant, we show that such graphs typically consist of many disjoint $d+1$-cliques and an almost triangle-free part. When $d$ is allowed to grow with $n$, we show that such graphs typically consist of $d+o(d)$ sized almost cliques together with an almost triangle-free part.
}

This confirms a conjecture of Collet and Eckmann from 2002 and considerably strengthens their observation that the triangles cannot be totally scattered in typical instances of regular graphs with many triangles.

\section{Introduction}

It is easy to see that a $d$-regular graph on $n$ nodes cannot have more than $\tm = \tm(n,d) = \frac{n}{3} \tbinom{d}{2}$ triangles, and that this value is achieved exactly if the graph is a disjoint union of $d+1$-cliques. In this paper we prove a \emph{strong, almost sure structural stability result} for this extremal problem: Let $0 < c \leq 1$ and let $\Gcal_{d,c}(n)$ denote the set of $d$-regular graphs on $n$ labeled nodes that contain at least $c \cdot \tm$ triangles. We show that, for constant $d$ and large $n$, almost all graphs in $\Gcal_{d,c}(n)$ consist of a disjoint union of $d+1$ cliques and an almost triangle-free part. 

This result may not seem surprising at first, and especially for $c$ close to 1 it seems quite natural to expect. However we would like to emphasize that it also holds for very small positive $c$. In that regime the required number of triangles could easily be arranged in such a way that no vertex has more than 1 triangle in its neighborhood. However, as it turns out, such graphs only constitute a vanishing fraction of elements of $\Gcal_{d,c}(n)$.

The study of $\Gcal_{d,c}(n)$ has been initiated by Collet and Eckmann~\cite{eckmann2002}, who proved
\begin{equation}\label{eq:eckmann}
1-c \frac{2 d}{d+1} \leq \liminf_{n\to \infty} \frac{\log \Gcal_{d,c}(n)}{\log \Gcal_d(n)} \leq \limsup_{n\to \infty} \frac{\log \Gcal_{d,c}(n)}{\log \Gcal_d(n)} \leq 1- \frac{c}{3},
\end{equation}
where $\Gcal_d(n)$ is the set of all $d$-regular graphs on $n$ labeled nodes. They conjectured that this limit exists. Using a heuristic counting argument they concluded that the triangles should be clustered in a ``typical'' element of $\Gcal_{d,c}(n)$. 

In this paper we prove their conjecture and find the correct limiting value
\begin{equation}
\lim_{n\to \infty} \frac{\log \Gcal_{d,c}(n)}{\log \Gcal_d(n)} = 1 - c \cdot  \frac{d-1}{d+1}.
\end{equation}
Furthermore we prove that the observation that triangles are clustered in typical elements of $\Gcal_{d,c}(n)$ hold in a very strong sense: for almost all graphs in the class, almost all triangles are contained in $d+1$-cliques. 

\subsection{Probabilistic context}

While our methods are purely combinatorial, the results are more conveniently stated in the language of random graphs. What is the probability that a random graph has a lot more triangles than expected? This is a typical question in the field of \emph{large deviations}, the theory that studies the tail behavior of random variables or, stated differently, the behavior of random objects conditioned on a parameter being far from its expectation. For example, one of the earliest results of this flavor, Cram\'er's Theorem states that for i.i.d. variables $X \sim X_1,X_2,\dots$ there exists a ``rate function'' $I(x)$ depending on the distribution of $X$ such that 
\[ \Prob{ \sum_1^N X_i \geq Nx} \approx e^{-N \cdot I(x)}.\] 

\subsubsection{Upper tail interpretation}

In random graphs, the question about the upper tail for triangles in $\Gbb(n,p)$ has been long studied for a constant factor of deviation from the mean \cite{janson2002infamous}.  More precisely, let $ t(\Gbb(n,p))$ denote the triangle density in the Erd\H{o}s-R\'{e}nyi random graph, normalized so that $\Exp{t(\Gbb(n,p))} = p^3$. One would like to understand the asymptotic behavior of 
\[ 
	r(n,p,\delta) = -\log \Prob{t(\Gbb(n,p) > (1+\delta)p^3}
\] 

The dense case ($p$ a constant) has been reduced to an analytic variational problem by Chatterjee and Varadhan~\cite{chatterjee2011large} using methods from graph limits. However, the solution of this variational problem is only known in certain parameter ranges (see~\cite{lubetzky2015replica} for details).  In the sparse ($p = o(1)$) regime the asymptotics $r(n,p,\delta) \approx n^2 p^2 \log(1/p)$ have been determined in a long series of papers by many authors~\cite{vu2001large,kim2004divide,janson2004deletion,janson2004upper,chatterjee2012missing,demarco2011upper}. The variational methods were extended to (part of) the sparse regime in~\cite{chatterjee2016nonlinear} and using this, Lubetzky and Zhao~\cite{lubetzky2017variational} found the exact asymptotics of $r(n,p,\delta)$ in the $n^{-1/42} \log n \leq p \ll 1$ range. Recently, Cook and Dembo~\cite{cook_dembo_2018} and Augeri~\cite{augeri_2018} extended it to the range $n^{-1/2} \ll p \ll 1$, and Harel, Mousset and Samotij~\cite{harel_2019} to all $n^{-1} \log n \ll p \ll 1$.

In the case of random regular graphs $\Gbb_d(n)$ have been studied much less. Kim, Sudakov, and Vu~\cite{kim2007small} obtained that the distribution of small subgraphs of $\Gbb_d(n)$ is asymptotically Poisson in the sparse case, implying an asymptotic formula for the tail probability $\Prob{T(\Gbb_d(n)) > (1+\delta)\Exp{T(\Gbb_d(n))}}$, where $T(G)$ denotes the number of triangles in the graph $G$. The expected number of triangles in $\Gbb_d(n)$ is finite. This means that the ``standard'' regime for large deviations - exceeding the expected value by a constant factor - in this case is not interesting.

\subsubsection{Maximum entropy random graphs with triangles}

In this note we are interested in the more extreme tail probability $\Prob{T(\Gbb_d(n)) > c \tbinom{d}{2}n/3 }$. The reason for analyzing this tail probability originated from a related problem of finding random graph models that maximize entropy under specific constraints. 

Let $\pr_{n}$ be some probability distribution on the set $\Gcal(n)$ of graphs on $n$ labeled nodes. Then the entropy of $\pr_{n}$ is defined as
\begin{equation}\label{eq:def_entropy_P}
	\Ecal[\pr_{n}] = \sum_{G \in \Gcal(n)} - \pr_{n}(G)\log\left(\pr_{n}(G_n)\right).
\end{equation}
Now let $\Gcal^\ast(n)$ denote the set of graphs on $n$ labeled nodes with some additional properties, for instance specified edge or triangle densities. Then, in order to study the structure of "typical" graphs with these constraints, one wants to find the uniform distribution on $\Gcal^\ast(n)$. This corresponds to finding the distribution $\pr_{n}^\ast$ that maximizes the entropy $\Ecal[\pr_{n}]$ subject to the constraint that $\pr_{n}^\ast = 0$ outside $\Gcal^\ast(n)$. 


It turns out that in many cases, computing the rate function also comes down to solving an optimization problem involving entropy. For example, Chatterjee and Dembo~\cite{chatterjee2016nonlinear} showed that, up to lower order terms, the rate function corresponding to the large deviation result for subgraph counting can be expressed as the solution to a specific entropy related optimization problem. For large deviations of triangles, let $\mathscr{G}_n$ denote the set of undirected graphs on $n$ nodes with edges weights $g_{ij} \in [0,1]$, then the rate function is obtained, up to lower order terms, as
\[
	r(n,p,\delta) = \inf\left\{I_p(G) \, : \, G \in \mathscr{G}_n, \, t(G) > (1+\delta)p^3\right\}.
\]
where $t(G) = n^{-3} \sum_{1 \le i,j,k \le n} g_{ij} g_{jk} g_{ki}$ and $I_p(G)$ is the so-called relative entropy of the weighted graph $G$
\[
	I_p(G) = \sum_{1 \le i < j \le n} g_{ij}\log \frac{g_{ij}}{p} + (1-g_{ij})\log \frac{1 - g_{ij}}{1-p}.
\]

In the case of dense graphs, such optimizations problems can be used to establish structural results for constrained random graphs. In the case of edge and triangle densities, a collection of research by Kenyon, Radin and co-authors~\cite{radin2013phase,radin2014asymptotics,kenyon2016bipodal,kenyon2017phases} showed that the limits of dense maximal entropy random graphs with given edge and triangle densities have a bipodal structure, at least in a narrow range just above the average. This means that the graph is split into two components with specific inter- and intra-component connection probabilities.

Recently, some techniques have been extended to solve the problem of finding maximum entropy sparse graphs with a given power-law degree distribution~\cite{hoorn2017sparse}. However, the degree distribution is a relatively global characteristic and hence is not expected to influence graph structures that much. A natural extension of this problem is therefore to include a constraint related to triangles, try to find the corresponding maximum entropy solutions and see what this tells us about the structure of such graphs. A key motivation for this kind of question is the work by Krioukov~\cite{krioukov2016clustering}, which hinted to the fact that triangle constraints might enforce the resulting maximum entropy solution to have some geometric component.

\subsection{Results}

Motivated by the question ``can local constraints induce global (geometric) behavior?'', we study the random $d$-regular graph $\Gbb_d(n)$ conditioned on having at least a positive fraction of the maximum possible number of triangles. (For $d$ fixed this just means linearly many triangles, in $n$.) 
With respect to the previous section, our setting is related to the entropy maximization problem with local and global constraints, i.e. where each node must have degree exactly $d$ and must be incident to at least $t$ triangles on average.

Let $\tm = \tm(n,d) = \tbinom{d}{2}n/3$ be the maximum number of triangles an $n$ vertex $d$-regular graph can have. Let $\Gcal_{d,c}(n)$ denote the set of $d$-regular graphs on $n$ labeled nodes that contain at least $c \cdot \tm$ triangles. We compute the leading asymptotics of $|\Gcal_{d,c}(n)|$ for fixed $c$, as $n \to \infty$, where $d$ is either a constant or can grow with $n$ as long as $\log d = o(\log n)$. We provide a structural description of a ``typical'' element of $\Gcal_{d,c}(n)$. We then extend these results to case of $k$-cliques in $d$-regular graphs.

\subsubsection{Number of $d$-regular graph with many triangles}

The dependence of $d$ on $n$ will be suppressed from the notation. We always assume $d = o(n)$. We will emphasize when constant $d$ is assumed.

\begin{theorem}\label{thm:counting} For a fixed $0 < c < 1$ we have 
\[ 
 - O\left(\frac{1}{\log \frac{n}{d}}\right) \leq \frac{\log \abs{\Gcal_{d,c}(n)}}{\frac{dn}{2} \log \frac{n}{d+1}} - \left(1 - c\cdot \frac{d-1}{d+1}\right) \leq  c\frac{\log d}{\log \frac{n}{d+1}} +  O\left(\frac{1}{\log \frac{n}{d}}\right)\] 
\end{theorem}

The part $ \frac{dn}{2} \log \frac{n}{d+1}$ is related to $\log \abs{\Gcal_{d}(n)}$, where $\Gcal_{d}(n)$ denotes the set of $d$-regular graphs on $n$ nodes. In particular, using the results in~\cite{wormald2017}, one can show that
\[
	\lim_{n \to \infty} \frac{\log \abs{\Gcal_{d}(n)}}{\frac{dn}{2} \log \frac{n}{d+1}} = 1
\]

The $O(1/\log(n/d))$ terms are $o(1)$ as long as $d = o(n)$. The $c \log d / \log(n/d)$ term on the right hand side is only $o(1)$ if $\log d = o(\log n)$. Unfortunately, for $d$ polynomial in $n$ we do not get a sharp logarithmic rate.

\subsubsection{Structure of $d$-regular graph with many triangles}

For fixed $d$, it turns out, perhaps not so surprisingly, that in most elements of $\Gcal_{d,c}(n)$, most of the triangles cluster into (disjoint) $d+1$-cliques. To make this statement precise, let us call a node \emph{bad} if it is not part of a $d+1$-clique but it is incident to at least one triangle. 

\begin{theorem}\label{thm:structure_first} Let $d$ be \emph{fixed} and $0<c<1$. With high probability a uniformly randomly chosen element of $\Gcal_{d,c}(n)$ has less than $ \frac{\log\log n}{ \log n}  n$ bad nodes. Thus, the number of triangles that are not part of a $d+1$-clique is sublinear.
\end{theorem}

In Section~\ref{ssec:structure_proofs} we prove a slightly more general result where we consider the case where a uniformly randomly chosen element of $\Gcal_{d,c}(n)$ has less than $ \varepsilon_n n$ bad nodes, with $\varepsilon_n \to 0$, such that $\varepsilon_n \log n \to \infty$. 

Note that Theorem~\ref{thm:structure_first} hints at a graph structure similar to the bipodal case, where instead of two components, we now have a linear in $n$ number of cliques and some remaining larger graph with a sub-linear number of triangles. 

We prove a similar result for the $1 \ll d \ll n$ case. Here, however, we cannot expect $d+1$-cliques to appear, as it is possible to construct families of examples with the correct leading logarithmic growth rate, that don't have any cliques. Instead, we introduce a notion of a pseudo-clique, which turns out to be a very dense subgraph of size $d+o(d)$ with the property that different pseudo-cliques must be disjoint. (See the explanation at the beginning of Section~\ref{sec:growing} for details.) It turns out that a typical element of the ensemble consists of a collection of these pseudo-cliques together with an almost triangle-free part.

\begin{theorem}\label{thm:pseudo_structure_first}
Let $1\ll d \ll n$ and fix $0 < c < 1$. With high probability, almost all triangles of a uniformly randomly chosen element of $\Gcal_{d,c}(n)$ are contained in pseudo-cliques. 
\end{theorem}

\subsubsection{$d$-regular graph with many $k$-cliques}

As a corollary to our methods, we also obtain similar results for regular graphs with many $k$-cliques. Let $\Gcal_{d,c,k}(n)$ denote the set of $d$-regular graphs on $n$ nodes that contain at least $c\cdot T_{k,\max} = c \tbinom{d}{k-1} n/k$ subgraphs isomorphic to $K_k$. As a natural extension of terminology, we call nodes \emph{bad} if they are not part of a $d+1$-clique but are incident to a $k$ clique.


\begin{theorem}\label{thm:k_clique}
For  $k \ge 3$ and fixed $0 < c < 1$ we have
\[  
	\abs*{\frac{\log \abs{\Gcal_{d,c,k}(n)}}{(d/2)n \log n} - \left(1 - c \cdot \frac{d-1}{d+1}\right)} = O(\log d /\log n)
\] 
Furthermore, \emph{for $d$ fixed}, almost all elements of $\Gcal_{d,t_k,k}(n)$ will have at most $\ep n$ bad nodes.
\end{theorem}

\section{Proofs}

\subsection{The number of regular graphs with a given number of triangles}

The proof of Theorem~\ref{thm:counting} consist of establishing a lower and upper bound on $\log|\Gcal_{d,c}(n)|$. More precisely, we will show that
\[
	  - \bigO{dn} \le \log|\Gcal_{d,c}(n)| - \left(1 - c\cdot \frac{d-1}{d+1}\right)\frac{dn}{2} \log \frac{n}{d+1}
	\le  c \frac{dn}{2} \log d +O(dn).
\]
The theorem then follows after dividing by $\frac{dn}{2} \log \frac{n}{d+1}$ and letting $n\to \infty$. 

\begin{proof}[Proof of Theorem~\ref{thm:counting} (Lower bound)]
To establish a lower bound we construct a family of elements in $\Gcal_{d,c}(n)$ by letting 
\[
	b =  c \cdot \tm \cdot \tbinom{d+1}{3}^{-1} =  \frac{c \cdot n}{d+1},
\]
taking $b$ disjoint $d+1$-cliques and an arbitrary $m = n-(d+1)b = (1-c)n$ node $d$-regular graph. Clearly, these graphs will have at least $c\cdot \tm$ triangles. 

The number of $d$-regular graphs on $m$ nodes satisfies
\[\abs{\Gcal_d(m)} \sim e^{1/4}\binom{m-1}{d}^m \binom{\tbinom{m}{2}}{\frac{dm}{2}}\binom{m(m-1)}{md}^{-1}\] for any $d = d(m) \le m - 2$, as $m \to \infty$, see~\cite{wormald2017}. Using the standard $(a/b)^b \leq \tbinom{a}{b} \leq (ea/b)^b$ bounds, it is easy to obtain

\begin{equation}\label{eq:gdn_asymptotic}
\log \abs{\Gcal_d(m)} \geq \frac{1}{2}dm\log\frac{m}{d+1} - dm - O(1).
\end{equation}
The size of our family of graphs thus satisfies
\[ \abs{\Gcal_{d,c}(n)} = \frac{\binom{n}{d+1}\binom{n-(d+1)}{d+1}\cdots \binom{n-(b-1)(d+1)}{d+1}}{b!} \abs{\Gcal_d(m)} = \frac{n!}{m! b! (d+1)!^b} \abs{\Gcal_d(m)}. \]
Again a simple computation using the $\abs{\log k! - (k \log k - k + 1/2 \log k)}  \leq  O(1)$ approximation, and noting that $n = m+ b(d+1) = (1-c)n + b(d+1)$, gives
\begin{align*}
	\log \abs{\Gcal_{d,c}(n)} 
	&\geq n \log n  - n +\frac{1}{2}\log n - (m \log m - m + \frac{1}{2}\log m) \\ 
	&- b((d+1) \log (d+1) - (d+1) + \frac{1}{2} \log (d+1)) - ( b \log b - b + \frac{1}{2} \log b) \\
	&+  \frac{d}{2}m\log m - \frac{d}{2} m \log d - dm \\
	&= cn\log\frac{m}{d+1} + \frac{1-c}{2} dn \log\frac{m}{d+1} - \frac{cn}{d+1} \log\frac{m}{d+1} - O(dn)\\
  	&= \left(c- \frac{c}{d+1} + \frac{(1-c)\cdot d}{2}\right)n\log \frac{m}{d+1} - O(dn) \\ 
	&= \left(\frac{2c}{d+1}  + 1- c\right)\frac{d}{2}n\log\frac{m}{d+1} -O(dn)\\ 
	& = \left(1- c\cdot \frac{d-1}{d+1}\right)\frac{d}{2}n \log\frac{m}{d+1} - \bigO{dn} \\
	& = \left(1- c\cdot \frac{d-1}{d+1}\right)\frac{d}{2}n \log\frac{n}{d+1} - \bigO{dn}. 
\end{align*}
We have used $b = \frac{n-m}{d+1} = \frac{cn}{d+1}$ and hence $\log b = \log \frac{m}{d+1} + O(1)$, and similarly $\log m = \log n + O(1)$.
\end{proof}

We now need to prove a matching upper bound on $\abs{\Gcal_{d,c}(n)}$. We do this by uncovering the edges of such graphs in a suitably chosen order, and recording whether in each step a new triangle is created. We will define a function \[\phi : \Gcal_{d,c}(n) \to \{0,1\}^{nd/2}\] that will record which edges of $G$ create triangles when added in this order.

We use an approach inspired by the configuration model. Let us denote by $\Gcals_d(n)$ (respectively, $\Gcals_{d,c}(n)$) the set of $d$-regular graphs (respectively, $d$-regular graphs with at least $c\cdot \tm$ triangles) on $n$ labeled nodes, where additionally the edges leaving each node are assigned labels 1 through $d$. This means that each edge gets two labels, one from each end. 
 
Given $\Gs \in \Gcals_{d,c}(n)$, we define the a \emph{configuration ordering} $\prec$ on the set of edges of $\Gs$ as follows. Let $e = (i_1 j_1)$ and $f = (i_2 j_2)$ be two edges of $\Gs$ with $i_1 < j_1$ and $i_2 < j_2$.  Let us declare $e \prec f$ if 
$i_1 < i_2$, or if $i_1 = i_2$ and the label of $e$ is smaller than the label of $f$ at their common node.  Let $e_1 \prec e_2 \prec \dots \prec e_{nd/2}$ denote the edges of $\Gs$ in increasing configuration order. Let $\Gs[k]$ denote the subgraph of $\Gs$ consisting of $e_1, \dots, e_k$. 

Finally define $\phi(\Gs)(k) = 1$ if $e_k$ is incident to a triangle in $\Gs[k]$ and 0 otherwise. Denoting $e_k = (ij)$, it is clear that we have $\phi(\Gs)(k) = 1$ if and only if there is a triangle $(hij)$ in $\Gs$ such that $h < \min(i,j)$. 

For any $x \in \{0,1\}^{nd/2}$ let us denote $\abss{x} = \sum_{j=1}^{nd/2} x(j)$. Then $\abss{\phi(\Gs)}$ denotes the total number of edges $e_{k}$ that upon adding to the graph $\Gs[k-1]$ have created at least one new triangle. Moreover, any vector $x \in \{0,1\}^{nd/2}$ describes a profile of which edges revealed a new triangle. The next lemma gives an upper bound on the number of graphs in $\Gcals_{d,c}(n)$ with a given triangle reveal profile.

\begin{lemma}\label{lem:phi_inverse}
\[ \abs*{ \phi^{-1}(x) } \leq  (dn)^{dn/2} \cdot  \left(\frac{d^2}{n}\right)^{\abss{x}}\]
\end{lemma}
\begin{proof}
The idea is to reconstruct a $\Gs \in \phi^{-1}(x)$ by starting from the empty graph and adding edges 1-by-1, according to the configuration order. Just like in the configuration model, each node starts with $d$ half-edges, labeled 1 through $d$. First we take the half-edge with label 1 at node 1, and join it to any other half-edge. We can do this in $dn -1$ ways. Then, in each subsequent step, we take the smallest node that still has half-edges, pick the one with the smallest label, and match it to any another half-edge. If we didn't have constraints on triangles, the total number of possible (multi-)graphs we could create this way would be $(dn-1)(dn-3)\cdots 3\cdot 1$, which is an upper bound on $|\Gcals_d(n)|$. 

In our case, the vector $x$ dictates whether the next edge added has to create a triangle with previously added edges. By the definition of the configuration order, the number of possible choices for the $k$th edge is $dn-(2k-1)$, as the starting half-edge is fixed and there are exactly $dn-(2k-1)$ available half-edges at this step. However, when $x(k) = 1$, the number of choices for the ending half-edge is limited. Suppose the starting half-edge is incident to node $j$. Then, in order for this edge to create a triangle, the ending half-edge most be incident to one of the current 2nd neighbors of $j$. There are never more than $d^2$ second neighbors, and thus never more than $d^3$ possible half-edges to choose from.

Thus we get the upper bound 
\begin{align*} \abs{\phi^{-1}(x)} & \leq \prod_{j: x(j) = 0} (dn-(2j-1)) \cdot \prod_{j: x(j)=1} d^3 
\\ & \leq d^{3\abss{x}}\cdot (dn)^{dn/2-\abss{x}} 
\\ & = (dn)^{dn/2} \cdot  \left(\frac{d^2}{n}\right)^{\abss{x}}\end{align*} which proves the lemma.
\end{proof}

The main idea for the upper bound is now to consider a specific set of triangle reveal profiles $x \in \{0,1\}^{nd/2}$, in which at least a $c\frac{d-1}{d+1}$ fraction of edges have revealed triangles.

\begin{proof}[Proof of Theorem~\ref{thm:counting} (Upper bound)]
Let us 
introduce the following short hand notation,
\begin{equation}\label{eq:tc} T_c = c\cdot \frac{dn}{2} \frac{d-1}{d+1},
\end{equation} 
as it will come up frequently. Define 
\[
	L = \left\{ x \in \{0,1\}^{\frac{nd}{2}} : \abss{x} \geq T_c - 1 \right\}. 
\]
Then, by Lemma~\ref{lem:phi_inverse}, and using $d^2 \leq n$, we see that 
\begin{equation}\label{eq:L_inverse_bound} 
	\abs*{\phi^{-1}(L)} \leq \abs{L} (dn)^{dn/2} \left(\frac{d^2}{n}\right)^{T_c - 1} \leq 2^{dn/2}(dn)^{dn/2} \left(\frac{d^2}{n}\right)^{T_c - 1}
\end{equation}
To finish the proof, we will show that $\abs*{\Gcals_{d,c}(n)} \leq \frac{dn}{2} \abs{\phi^{-1}(L)}$. 
For this, consider the symmetric group $S_n$, which acts on $\Gcals_{d,c}(n)$ by permuting the node labels. For $\sigma \in S_n$ and $\Gs \in \Gcals_{d,c}(n)$, let us denote by $\Gs_\sigma$ the graph obtained by applying $\sigma$ to the node labels. Furthermore let $S_n\Gs = \{\Gs_\sigma : \sigma \in S_n\} \subset \Gcals_{d,c}(n)$ denote the orbit of $\Gs$ under the action of $S_n$. We finish the proof modulo the following result, which we establish at the end of this section.

\begin{lemma}\label{lem:orbit_bound}
For any $\Gs \in \Gcals_{d,c}(n)$ we have 
\[ 
	\abs*{ S_n \Gs \cap \phi^{-1}(L) } \geq \frac{2}{dn} \abs{ S_n \Gs } 
\]
In other words, randomly relabeling the nodes of $G^*$ yields, with not too small probability, a graph whose $\phi(G^*_\sigma) \geq T_c-1$.
\end{lemma}

Summing this inequality over all orbits of the $S_n$ actions yields $\abs*{\Gcals_{d,c}(n)} \leq \frac{dn}{2} \abs{\phi^{-1}(L)}$ as claimed above. Note that $\abs{\Gcals_{d,c}(n)} = \abs{\Gcal_{d,c}(n)} \cdot (d!)^n$. Combining this with \eqref{eq:L_inverse_bound} we get
\begin{align*}
\abs*{\Gcal_{d,c}(n)} &= \frac{\abs{\Gcals_{d,c}(n)}}{(d!)^n} \leq \frac{dn}{2} \frac{\abs{\phi^{-1}(L)}}{(d!)^n} \leq \frac{dn}{2} 2^{dn/2} \frac{ (dn)^{dn/2}}{(d/e)^{dn}} \left(\frac{d^2}{n}\right)^{T_c - 1}
\\ &= \frac{dn}{2} (\sqrt{2}e)^{dn} \left(\frac{n}{d}\right)^{\frac{dn}{2} - T_c + 1} d^{T_c-1}
\\ &= (\sqrt{2}e)^{dn} \left(\frac{n}{d}\right)^{\frac{dn}{2} - T_c} d^{T_c} \frac{n^2}{2d}
\end{align*}
Thus 
\begin{align*} \log \abs{\Gcal_{c,d}(n)} &\leq \left(1-c \cdot \frac{d-1}{d+1}\right)\frac{dn}{2} \log\frac{n}{d} + c\cdot \frac{dn}{2} \log d + O(dn)
\\ &= \left(1-c \cdot \frac{d-1}{d+1}\right)\frac{dn}{2} \log\frac{n}{d+1} + c\cdot \frac{dn}{2} \log d + O(dn)
\end{align*}
\end{proof}

We are thus left to prove Lemma~\ref{lem:orbit_bound}. For this we first show that for a uniform random permutation $\sigma$, the expected value of $\abss{\phi(\Gs_\sigma)}$ is at least $c\cdot \frac{dn(d-1)}{2(d+1)}$. Then the lemma will follow from a standard Markov-inequality argument.

\begin{lemma}\label{lem:expected_triangle_count}
Let $\sigma$ be a uniformly random permutation. Then
\[ \Exp{\abss{\phi(\Gs_\sigma)}} \geq T_c.\]
\end{lemma}

\begin{proof}

Let $X_e(\sigma)$ be the indicator variable 
that the edge $e$ of $\Gs$ 
creates a triangle when it is added in the lexicographic order of $\Gs_\sigma$. Then 
$\abss{\phi(\Gs_\sigma)} = \sum_e X_e(\sigma)$ and so 
\[\Exp{\abss{\phi(\Gs_\sigma)}} = \sum_e \Exp{X_e(\sigma)} = \sum_e \Prob{X_e(\sigma)=1} \]
Let $e = (ij)$ and let $e$ be incident to exactly $t_e$ triangles in $\Gs$. Let $v_1, v_2, \dots, v_{t_e}$ denote the third nodes of these triangles. Then $X_e(\sigma)$ is 1 if at least one of these triangles 
is formed at the moment when 
$e$ is added, which is equivalent to at least one of these nodes preceding both $i$ and $j$ in the $\sigma$-order. That is, $\min(\sigma(v_1), \sigma(v_2), \dots, \sigma(v_{t_e})) < \min(\sigma(i), \sigma(j))$. Then $X_e(\sigma) = 0$ if and only if either $i$ or $j$ has the smallest $\sigma$ value among $i,j, v_1, v_2, \dots, v_{t_e}$. Since the $\sigma$-order of these nodes is a uniformly random permutation on $t_e+2$ elements, we get $\Prob{X_e(\sigma) = 0} = 2/(t_e+2)$ and hence $\Prob{X_e(\sigma)=1} = 1-2/(t_e +2)$. 

Thus, since $t_e \leq d-1$, we get
\begin{equation}\label{eq:expected_moments} \Exp{\abss{\phi(\Gs_\sigma)}} = \sum_e \Prob{X_e(\sigma)=1}  = \sum_e \left(1 - \frac{2}{t_e+2}\right) = \sum_e \frac{t_e}{t_e+2} \geq \sum_e \frac{t_e}{d+1} \geq T_c,\end{equation} where the last inequality follows from $\sum_e t_e$ being 3 times the total number of triangles in $\Gs$, which is in turn at least $c \cdot \frac{n}{3} \tbinom{d}{2}$.  This finishes the proof of the lemma.
\end{proof}

\begin{proof}[Proof of Lemma~\ref{lem:orbit_bound}] By simple algebraic considerations 
\begin{equation}\label{eq:alg} \frac{\abs*{ S_n \Gs \cap \phi^{-1}(L) } }{\abs{ S_n \Gs } } = \frac{\abs{\{ \sigma \in S_n : \phi(\Gs_\sigma) \in L \}}}{\abs{S_n}}.\end{equation} This is obvious when $\Gs$ has no automorphisms (that is, when $S_n \Gs$ is in bijection with $S_n$), but it also holds in the general case since the stabilizers of different elements of the orbit $S_n \Gs$  are conjugate and hence have the same cardinality.

Consider a uniformly random permutation $\sigma \in S_n$. By \eqref{eq:alg} it is enough to show that with probability at least $\frac{2}{dn}$ we have $\phi(\Gs_\sigma) \in L$, which is equivalent to $\abss{\phi(\Gs_\sigma)} \geq T_c-1$. 

Observe that $\abss{\phi(\Gs_\sigma)}$ cannot be bigger than $\frac{dn}{2}$. Hence, using Lemma~\ref{lem:expected_triangle_count}
\begin{equation}\label{eq:markov}
\begin{aligned}
	T_c \leq \Exp{\abs{\phi(\Gs_\sigma)}} &\leq \left(T_c-1\right)\Prob{\abs{\phi(\Gs_\sigma)} < T_c-1} + \frac{dn}{2}\Prob{\abs{\phi(\Gs_\sigma)} \geq T_c-1} \\
	&\leq T_c-1 + \frac{dn}{2}\Prob{\abs{\phi(\Gs_\sigma)} \geq T_c-1},
\end{aligned}
\end{equation}
from which we conclude that
\[ \Prob{\abss{\phi(\Gs_\sigma)}  \geq  T_c-1} \geq \frac{2}{dn}. \]

\end{proof}

\subsection{The structure of regular graphs with a given number of triangles}\label{ssec:structure_proofs}

A simple extension of the methods of the proof of Theorem~\ref{thm:counting} yields a strong structural description of a typical graph with at least $c \cdot \tm$ triangles: 
\begin{description}
\item[For $d$ fixed:] $1-o(1)$ fraction of all triangles are contained in $d+1$-cliques.
\item[For $1 \ll d \ll n$:] $1-o(1)$ fraction of all triangles are contained in pseudo-cliques. Moreover, these pseudo-cliques are non-overlapping.
\end{description}

We will treat the two cases separately, but the following lemma will be useful for both. As before, we let $t_e$ denote the number of triangles the edge $e$ is incident to. We say the edge $e$ is \emph{$\delta$-bad} if $1 \leq t_e \leq d-1-\delta d$. 

\begin{lemma}\label{lem:badness}
Let $\ep, \delta >0$ fixed. Let $\Gcal_{d,c}^{\ep, \delta} \subset \Gcal_{d,c}$ denote the subset of graphs where at least $\ep (d/2)n$ edges are $\delta$-bad. Then
\[ \log \abs{\Gcal_{d,c}^{\ep,\delta}(n)} \leq  \left(1 - c\frac{d-1}{d+1} - \frac{\ep \delta d}{3d+3}\right) \frac{dn}{2} \log \frac{n}{d} + \left(c + \frac{\ep \delta d}{3d+3}\right)\frac{dn}{2} \log d + O(dn).\]
\end{lemma} 

\begin{proof}
If $e$ is bad, then $1 \leq t_e \leq d-1-\delta d$, so 
\[ \frac{t_e}{t_e+2} = \frac{t_e}{d+1} + \left( \frac{t_e}{t_e+2} - \frac{t_e}{d+1}\right) = \frac{t_e}{d+1} + \frac{t_e}{t_e+2}\cdot \frac{d - 1-t_e}{d+1}\geq \frac{t_e}{d+1} + \frac{1}{3}\cdot \frac{\delta d}{d+1}.\]
Suppose more than $\ep (d/2)n$ edges of $\Gs \in \Gcals_{d,c}(n)$ are bad.  Combining 
the above with \eqref{eq:expected_moments} we get that for a uniformly random permutation $\sigma \in S_n$ 
\[ \Exp{\abs{\phi(\Gs_\sigma)}} = \sum_e \frac{t_e}{t_e+2} \geq \sum_e \frac{t_e}{d+1} + \frac{\ep \delta d^2 n}{6d+6}  \geq T_c + \frac{\ep \delta d^2 n}{6d+6}.\]
Hence, by the same computation as in \eqref{eq:markov} we get
\[ \Prob{\abs{\phi(\Gs_\sigma)} \geq T_c + \frac{\ep \delta d^2 n}{6d+6} -1} \geq \frac{2}{dn}.\]
Now let 
\[L_{\ep,\delta} = \left\{ x \in \{0,1\}^{\frac{nd}{2}} : \abs{x} \geq T_c + \frac{\ep \delta d^2 n}{6d+6} - 1 \right\}. \]
By the previous considerations, for any $\Gs \in {\Gcals}^{\ep,\delta}_{d,c}(n)$ we get that
\[ \frac{\abs*{ S_n \Gs \cap \phi^{-1}(L_{\ep,\delta}) } }{\abs{ S_n \Gs } } = \frac{\abs{\{ \sigma \in S_n : \phi(\Gs_\sigma) \in L_{\ep,\delta} \}}}{\abs{S_n}} \geq \frac{2}{dn}\]
Summing the inequality $\abs*{ S_n \Gs \cap \phi^{-1}(L_{\ep,\delta}) } \geq \frac{2}{dn} \abs{S_n}$ over the orbits of the $S_n$ action in ${\Gcals}^{\ep,\delta}_{d,c}(n)$ we obtain the estimate 
\[ \abs{{\Gcals}^{\ep,\delta}_{d,c}(n)} \leq \frac{dn}{2} \abs{\phi^{-1}(L_{\ep,\delta})} ,\] which, combined with Lemma~\ref{lem:phi_inverse}, yields
\begin{align*} 
\abs{\Gcal^{\ep,\delta}_{d,c}(n)} &= \frac{\abs{{\Gcals}^{\ep,\delta}_{c,d}(n)}}{(d!)^n} \leq \frac{dn}{2} \frac{\abs{\phi^{-1}(L_{\ep,\delta})}}{(d!)^n} 
\\ &\leq \frac{dn}{2} 2^{dn/2} \frac{ (dn)^{dn/2}}{(d/e)^{dn}} \left(\frac{d^2}{n}\right)^{T_c + \frac{\ep \delta d^2 n}{6d+6} - 1}
\\ &= \frac{dn}{2} (\sqrt{2}e)^{dn} \left(\frac{n}{d}\right)^{\frac{dn}{2} - T_c + 1} d^{T_c-1}
\\ &= (\sqrt{2}e)^{dn} \left(\frac{n}{d}\right)^{\frac{dn}{2} - T_c - \frac{\ep \delta d^2 n}{6d+6}} d^{T_c + \frac{\ep \delta d^2 n}{6d+6}} \cdot \frac{n^2}{2d}.
\end{align*}
Taking $\log$ of both sides finishes the proof.
\end{proof}

\subsubsection{Fixed $d$}

Let us say that a node in $G$ is bad if it's not in a $d+1$-clique, but it is in a triangle. The following statement is a (very) slight strengthening of Theorem~\ref{thm:structure_first}.

\begin{theorem}\label{thm:structure}
Let $\ep > 0$ and $d$ fixed. Among all $d$-regular graphs with at least $c \cdot \tm$ triangles, the proportion of those where more than $\ep n$ nodes are bad goes to 0 as $n\to \infty$. This remains true even if $\ep \to 0$, as long as $\ep \log n \to \infty$.
\end{theorem}

We will make use of the following simple observation, whose proof we omit.

\begin{lemma}\label{lem:max_triangles}
Let $G$ be a $d$-regular graph. If all edges incident to a node $v$ are incident to exactly $d-1$ triangles, then $v$ is part of a $d+1$-clique.
\end{lemma}

\begin{proof}[Proof of Theorem~\ref{thm:structure}]
Let us set $\delta = 1/d$ and call $1/d$-bad edges simply ``bad''. Suppose now that more than $\ep n$ nodes of $G$ are bad. Each bad node, by definition, is adjacent to at least two bad edges, so there are at least $\ep n$ bad edges. Thus $G \in \Gcal^{\frac{2\ep}{d} , \frac{1}{d}}_{d,c}(n)$.
Then, Lemma~\ref{lem:badness}  combined with Theorem~\ref{thm:counting} and the fact that $d = O(1)$ gives
\[ \log \frac{ \abs{\Gcal^{\frac{2\ep}{d} , \frac{1}{d}}_{d,c}(n)}}{\abs{\Gcal_{d,c}(n)}} = -\frac{\frac{2\ep}{d} \frac{1}{d} d}{3d+3} \frac{dn}{2} \log n + O(dn\log d)= -\frac{\ep}{3d+3}n\log n +O(n), \]
so indeed
\[ \lim_{n\to \infty}  \frac{ \abs{\Gcal^{\frac{2\ep}{d} , \frac{1}{d}}_{d,c}(n)}}{\abs{\Gcal_{d,c}(n)}} = 0,\] as long as $\ep \log n \to \infty$, proving that with high probability a graph conditioned on having at least $c \cdot \tm$ triangles has $o(n)$ bad nodes, hence consists almost completely of $d+1$-cliques and a triangle-free part.
\end{proof}

\subsubsection{Growing $d$}\label{sec:growing}

An immediate generalization of Theorem~\ref{thm:structure} cannot hold for the $d \gg 1$ case, because one can exhibit a family of $d$-regular graphs with $c\cdot \tm$ triangles that contain no cliques at all, yet have the optimal, $(1-c)(d/2)n\log \frac{n}{d+1}$, logarithmic growth rate. Such a family can be built, for example, by taking the disjoint union of many copies of $H$, together with a random $d$-regular graph, where $H$ is $K_{d+2}$ minus a perfect matching. Realizing the required $c\cdot \tm$ triangles takes up only slightly more space this way than using copies of $K_{d+1}$, and the resulting decrease in the size of the random part is small enough that it doesn't affect the logarithmic growth rate. One can push this even further, and use disjoint $d+o(d)$ size components (these still contain roughly $\tbinom{d}{3}$ triangles each), and a large random $d$-regular part of the appropriate size. 

We will show in this section, that a typical graph in the ensemble does, in fact, resemble an element of this last family. The main reason the previous argument fails for $d \gg 1$ is because now we cannot choose $\delta$ to be too small in Lemma~\ref{lem:badness}, otherwise the gain will be less in magnitude than the error term $O(dn\log d)$. Nevertheless, if $\log d / \log n$ is small, then the gap between the main term and the error term allows us to choose both $\ep$ and $\delta$ to be small, which will be enough to learn something about the typical graphs in the ensemble. In particular, we can choose
\begin{equation}\label{eq:epde} \ep = \delta^2 = (3c)^{2/3}\cdot \left(\frac{\log d}{\log \frac{n}{d^2}}\right)^{1/2}.\end{equation}

Then Lemma~\ref{lem:badness} implies that in a typical $d$-regular graph with at least $c\cdot \tm$ triangles, most edges are incident to 0 or almost $d$ triangles. As it turns out, this implies a structural description similar to that of Theorem~\ref{thm:structure}. Let us first informally explain the result. We call a subgraph $H \subset G$ a \emph{dense spot} if $|H| \leq d+1$ and $\deg_H(x) = d(1-O(\delta))$ for all $x \in H$. Dense spots satisfy the following simple, combinatorial observations:
\begin{itemize}
\item Two dense spots are either disjoint, or they intersect in $d(1-O(\delta))$ nodes. 
\item Intersection is transitive: if $H_1 \cap H_2 \neq \emptyset$ and $H_2 \cap H_3 \neq \emptyset$ then $H_1 \cap H_3 \neq \emptyset$.
\item The union of a maximal, pairwise intersecting, family of dense spots has size $d(1+O(\delta))$. We call these \emph{pseudo-cliques}. 
\item It follows that any two pseudo-cliques must be disjoint.
\end{itemize}

The following is a restatement of Theorem~\ref{thm:pseudo_structure_first}.

\begin{theorem}\label{thm:pseudo_structure}
Let $1 \ll d \ll n$. Let $\delta$ as in \eqref{eq:epde}, and assume $\delta < 1/16$. With high probability, a random $d$-regular graph with at least $c\cdot \tm$ triangles contains $(1+O(\delta))c n/d$ pseudo-cliques. These pseudo-cliques contain $1-O((\ep+\delta)/c)$ fraction of all triangles. 
\end{theorem}

\begin{remark}
Theorem~\ref{thm:pseudo_structure} is the strongest when $\log d = o(\log n)$, as in this case both $\ep$ and $\delta$ are $o(1)$. However,  when $d = n^\beta$ then $\delta = (3c)^{1/3} \left( \frac{\beta}{1-2\beta}\right)^{1/4}$, so we still get a non-trivial structural result when $\beta$ is small enough.
\end{remark}

\begin{proof}
We set $\ep$ and $\delta$ according to \eqref{eq:epde}. Then, a careful calculation using Lemma~\ref{lem:badness} shows that we have 
\[\lim_{n\to \infty}  \frac{ \abs{\Gcal^{\ep,\delta}_{d,c}(n)}}{\abs{\Gcal_{d,c}(n)}}=0,\] so it is enough to consider a graph $G \in G_{d,c}(n) \setminus G_{d,c}^{\ep, \delta}(n)$.  The graph $G$ then has, by definition, less than $\ep (d/2)n$ edges that are $\delta$-bad.  Let us call a $\delta$-bad edge \emph{bad} for brevity, and other edges \emph{good}. Let us start by removing all edges with $t_e = 0$ from $G$, and denote the remaining graph by $G'$. Removing such edges doesn't change the $t_e$ value of the remaining edges. Let us call a node $v \in G'$ bad, if it is incident to at least $\delta d$ bad edges. Then, since $\ep = \delta^2$, it follows that $G'$ cannot have more than $\delta n$ bad nodes. 

The total number of triangles that are incident to either a bad edge or a bad node is at most $\ep (d/2)n \cdot d + \delta n \tbinom{d}{2} = O(\ep+\delta) \cdot \tm$. We will show that the rest of the triangles are concentrated in pseudo-cliques.

\begin{definition}
A subgraph $H \subset G$ is a \emph{dense spot} if $|H| \leq d+1$ and each node $x \in H$ has $\deg_H(x) \geq (1-4\delta)d$.
\end{definition}

\begin{claim} Let $H_1,H_2$ be dense spots. Then they are either disjoint, or $\abs{H_1 \cap H_2} \geq (1-8\delta)d$.  This follows from the fact the nodes in the intersection must have degree $\leq d$.
\end{claim}

\begin{claim} Let $H_1,H_2,H_3$ be dense spots. If $H_1 \cap H_2 \neq \emptyset$ and $H_2 \cap H_3 \neq \emptyset$ then $H_1 \cap H_3 \neq \emptyset$, since otherwise we would have $d+1 \geq |H_2| \geq |H_2\cap H_1| + |H_2 \cap H_3| \geq  2d(1-8\delta)$ which contradicts $\delta < 1/16$. 
\end{claim}

\begin{definition}
A subgraph $K \subset G$ is a \emph{pseudo-clique} if there is a maximal family $\mathcal{H}$ of pairwise intersecting dense spots
such that $K = \cup_{H \in \mathcal{H}} H$.
\end{definition}
 
\begin{claim} By definition, any dense spot $H$ is either disjoint from, or fully contained in, a pseudo clique $K$. Furthermore, any two distinct pseudo-cliques are disjoint.
\end{claim}

 \begin{lemma}\label{lem:pseudo-clique} If $K$ is a pseudo-clique then $|K| \leq \frac{1 - 8\delta}{1- 13\delta} (d+1) = (1+O(\delta))d$.
 \end{lemma}

\begin{proof}[Proof of Lemma~\ref{lem:pseudo-clique}]
Let $H \subset K$ be one of the dense spots in $K$. For any node $x\in H$ we have $\deg_H(x) \geq (1-4\delta)d$. But $\deg(x) = d$, thus the total number of edges going between $H$ and $K \setminus H$ is at most $|H| \cdot 4\delta d \leq 4\delta d(d+1)$. However, each node $y \in K\setminus H$ is contained in a dense spot $H'$, and thus $\deg_{H'}(y) \geq (1-4\delta)d$. Since $\abs{H' \setminus H} \leq 8\delta d + 1 \le 9\delta d$, we get that at least $(1-13\delta)d$ edges go from $y$ to $H$. Hence \[|K \setminus H|(1-13\delta)d \leq  4\delta d(d+1),\] from which \[ |K| \leq |H| + \frac{4\delta(d+1)}{1- 13\delta} \leq \frac{1 - 8\delta}{1- 13\delta} (d+1) \] as claimed.
\end{proof}

To finish the proof of Theorem~\ref{thm:pseudo_structure}, we need to show that any triangle that's only incident to good edges and good nodes is contained in a pseudo-clique. We will show slightly more: that a good edge connecting good nodes is in a pseudo-clique. 

Consider a good edge $uv$ in $G'$, where both $u$ and $v$ are good nodes. Since we already removed the edges with no triangles, $t_{uv} \geq d-\delta d$. In particular $u$ and $v$ share at least $d-\delta d$ common neighbors. Each of $u$ and $v$ may be incident to at most $\delta d$ bad edges. That means that the subset $H_0$ of common neighbors of $u$ and $v$ that are connected to both of them via good edges has size $|H_0| \geq d-3\delta d$. Let $H = H_0 \cup \{u,v\}$. We claim $H$ is a dense spot. Clearly $|H| \leq 1+\deg(u) = d+1$, and by construction, $\deg_H(u), \deg_H(v) \geq (1-3\delta)d \geq (1-4\delta)d$. What remains to show is that for any node $x \in H_0$ we have $\deg_H(x) \geq (1-4\delta)d$. But 
$xu$ is a good edge, hence $x$ and $u$ have at least $(1-\delta)d$ common neighbors, or equivalently, at most $\delta d $ of $u$'s neighbors are not connected to $x$. Thus $x$ is connected to at least $(1-4\delta)d$ nodes in $H$, proving that indeed $H$ is a dense spot. So the $uv$ edge is contained in a dense spot, and thus in a pseudo-clique. 
\end{proof}
\subsection{$k$-cliques}\label{sec:k_clique}

We can easily extend the above results from triangles to $k$-cliques. Let $\Gcal_{d,c,k}(n)$ denote the set of $d$-regular graphs on $n$ nodes that contain at least $c \cdot \tbinom{d}{k-1} \frac{n}{k}$ subgraphs isomorphic to  $K_k$. (The maximum possible number of subgraphs isomorphic to $K_k$ is clearly $ \tbinom{d}{k-1} \frac{n}{k}$.)


\begin{proof}[Proof of Theorem~\ref{thm:k_clique}]
The idea is a simple reduction the the $k=3$ case. Clearly, each $G \in \Gcal_{d,c,k}(n)$ has at least \[ c \cdot \tbinom{d}{k-1} \frac{n}{k} \frac{\tbinom{k}{3}}{\tbinom{d-2}{k-3}} = c \cdot \tbinom{d}{2}\frac{n}{3} = c \cdot \tm \]  triangles, so $\Gcal_{d,t_k,k}(n) \subset \Gcal_{d,c}(n)$, which implies the upper bound of the theorem. 
On the other hand, the family of graphs constructed in Theorem~\ref{thm:counting} contain 
\[b \binom{d+1}{k} =  c \cdot \frac{n}{d+1}\binom{d+1}{k} = c\cdot \binom{d}{k-1} \frac{n}{k}\] $k$-cliques, so this family is contained in $\Gcal_{d,c,k}(n)$, implying the lower bound of the theorem. Finally, the structural statement follows directly from Theorem~\ref{thm:structure}.
\end{proof}

\paragraph{Acknowledgements}
The authors thank Dmitri Krioukov for useful discussions on the related topic of sparse maximum entropy graphs with given number of triangles, which lead us to the upper tail problem. Pim van der Hoorn and Gabor Lippner were supported by ARO grant W911NF1610391, Gabor Lippner was also supported by NSF grant DMS 1800738, and Elchanan Mossel was supported by NSF grant DMS-1737944 and ONR grant N00014-17-1-2598.

\bibliographystyle{plain}
\bibliography{references}

\begin{thebibliography}{10}

\bibitem{wormald2017}
L.~Anita and N.~Wormald.
\newblock Asymptotic enumeration of graphs by degree sequence, and the degree
  sequence of a random graph.
\newblock {\em arXiv}, 2017.

\bibitem{augeri_2018}
Fanny Augeri.
\newblock Nonlinear large deviation bounds with applications to traces of
  wigner matrices and cycles counts in erd{\"o}s-renyi graphs.
\newblock {\em arXiv preprint arXiv:1810.01558}, 2018.

\bibitem{chatterjee2012missing}
Sourav Chatterjee.
\newblock The missing log in large deviations for triangle counts.
\newblock {\em Random Structures \& Algorithms}, 40(4):437--451, 2012.

\bibitem{chatterjee2016nonlinear}
Sourav Chatterjee and Amir Dembo.
\newblock Nonlinear large deviations.
\newblock {\em Advances in Mathematics}, 299:396--450, 2016.

\bibitem{chatterjee2011large}
Sourav Chatterjee and S.R.S. Varadhan.
\newblock The large deviation principle for the erd{\H o}s-r{\'e}nyi random
  graph.
\newblock {\em European Journal of Combinatorics}, 32(7):1000 -- 1017, 2011.
\newblock Homomorphisms and Limits.

\bibitem{eckmann2002}
P.~Collet and J.~P. Eckmann.
\newblock The number of large graphs with a positive density of triangles.
\newblock {\em Journal of Statistical Physics}, 109(5):923--943, 2002.

\bibitem{cook_dembo_2018}
Nicholas~A Cook and Amir Dembo.
\newblock Large deviations of subgraph counts for sparse erd{\H o}s--r{\'e}nyi
  graphs.
\newblock {\em arXiv preprint arXiv:1809.11148}, 2018.

\bibitem{demarco2011upper}
B.~DeMarco and J.~Kahn.
\newblock Upper tails for triangles.
\newblock {\em Random Structures {\&} Algorithms}, 40(4):452--459, aug 2011.

\bibitem{harel_2019}
Matan Harel, Frank Mousset, and Wojciech Samotij.
\newblock Upper tails via high moments and entropic stability.
\newblock {\em arXiv preprint arXiv:1904.08212}, 2019.

\bibitem{janson2004upper}
Svante Janson, Krzysztof Oleszkiewicz, and Andrzej Ruci{\'n}ski.
\newblock Upper tails for subgraph counts in random graphs.
\newblock {\em Israel Journal of Mathematics}, 142(1):61--92, 2004.

\bibitem{janson2002infamous}
Svante Janson and Andrzej Ruci{\'n}ski.
\newblock The infamous upper tail.
\newblock {\em Random Structures \& Algorithms}, 20(3):317--342, 2002.

\bibitem{janson2004deletion}
Svante Janson and Andrzej Ruci{\'n}ski.
\newblock The deletion method for upper tail estimates.
\newblock {\em Combinatorica}, 24(4):615--640, 2004.

\bibitem{kenyon2016bipodal}
Richard Kenyon, Charles Radin, Kui Ren, and Lorenzo Sadun.
\newblock Bipodal structure in oversaturated random graphs.
\newblock {\em International Mathematics Research Notices}, page rnw261, 2016.

\bibitem{kenyon2017phases}
Richard Kenyon, Charles Radin, Kui Ren, and Lorenzo Sadun.
\newblock The phases of large networks with edge and triangle constraints.
\newblock {\em Journal of Physics A: Mathematical and Theoretical},
  50(43):435001, 2017.

\bibitem{kim2007small}
Jeong~Han Kim, Benny Sudakov, and Van Vu.
\newblock Small subgraphs of random regular graphs.
\newblock {\em Discrete Mathematics}, 307(15):1961--1967, 2007.

\bibitem{kim2004divide}
Jeong~Han Kim and Van~H Vu.
\newblock Divide and conquer martingales and the number of triangles in a
  random graph.
\newblock {\em Random Structures \& Algorithms}, 24(2):166--174, 2004.

\bibitem{krioukov2016clustering}
Dmitri Krioukov.
\newblock Clustering implies geometry in networks.
\newblock {\em Physical review letters}, 116(20):208302, 2016.

\bibitem{lubetzky2015replica}
Eyal Lubetzky and Yufei Zhao.
\newblock On replica symmetry of large deviations in random graphs.
\newblock {\em Random Structures \& Algorithms}, 47(1):109--146, 2015.

\bibitem{lubetzky2017variational}
Eyal Lubetzky and Yufei Zhao.
\newblock On the variational problem for upper tails in sparse random graphs.
\newblock {\em Random Structures \& Algorithms}, 50(3):420--436, 2017.

\bibitem{radin2014asymptotics}
Charles Radin, Kui Ren, and Lorenzo Sadun.
\newblock The asymptotics of large constrained graphs.
\newblock {\em Journal of Physics A: Mathematical and Theoretical},
  47(17):175001, 2014.

\bibitem{radin2013phase}
Charles Radin and Lorenzo Sadun.
\newblock Phase transitions in a complex network.
\newblock {\em Journal of Physics A: Mathematical and Theoretical},
  46(30):305002, 2013.

\bibitem{hoorn2017sparse}
Pim \swap{Hoorn}{~van~der~}, Gabor Lippner, and Dmitri Krioukov.
\newblock Sparse maximum-entropy random graphs with a given power-law degree
  distribution.
\newblock {\em Journal of Statistical Physics}, 173(3):806--844, Nov 2018.

\bibitem{vu2001large}
Van~H Vu.
\newblock A large deviation result on the number of small subgraphs of a random
  graph.
\newblock {\em Combinatorics, Probability and Computing}, 10(1):79--94, 2001.

\end{thebibliography}

\end{document}